\documentclass[12pt]{amsart}

\usepackage{longtable,latexsym,amssymb,amsmath}
\usepackage[latin1]{inputenc}
\newtheorem{thm}{Theorem}

\newtheorem{cor}[thm]{Corollary}
\newtheorem{prop}[thm]{Proposition}

\title{A diophantine equation for sums of consecutive like powers }
\author{Simon Felten and Stefan M\"uller-Stach}
\address{Institut f\"ur Mathematik, Johannes Gutenberg Universit\"at Mainz}
\email{sfelten@students.uni-mainz.de, mueller-stach@uni-mainz.de}
\date{}

\begin{document}

\begin{abstract}
We show that the diophantine equation
\[
n^\ell+(n+1)^\ell + \cdots + (n+k)^\ell=(n+k+1)^\ell+ \cdots + (n+2k)^\ell 
\]
has no solutions in positive integers $k,n$ for all $\ell \ge 3$.
\end{abstract}

\maketitle

\section{Introduction}

The reader may have seen the following triangular pattern of equations for sums of consecutive positive integers:

\begin{align*}
1+2 & =  3 \\
4+5+6 & =  7+8 \\ 
9+10+11+12 &= 13+14+15\\ 
& \vdots 
\end{align*}

The numbers on the far left in each line are precisely all square numbers $n=k^2$ for $k \ge 1$. Furthermore, 
$k$ is equal to the number of summands on the right, and one less than the number of summands on the left. 
There is a similar sequence for squares, which was already studied by Georges Dostor \cite{dostor} in 1879: 

\begin{align*}
3^2+4^2 & =  5^2 \\
10^2+11^2+12^2 & =  13^2+14^2 \\
21^2+22^2+23^2+24^2 & =  25^2+26^2+27^2 \\
& \vdots
\end{align*}

In this case, the numbers on the far left are of the form $n^2$ with $n=k(2k+1)$ for $k \ge 1$.
There are geometric explanations for these sums via square dissections, see \cite{boardman}.  
One immediately wonders whether the pattern persists for cubes and higher powers. 
This leads to the diophantine equation

\begin{equation} \label{gleichung}
n^\ell+(n+1)^\ell + \cdots + (n+k)^\ell=(n+k+1)^\ell+ \cdots + (n+2k)^\ell 
\end{equation}

with unknowns $k$ and $n$. This equation is similar to other diophantine equations for 
sums of like powers \cite[page 209]{guy}, in particular the Erd\"os-Moser equation. 
For $\ell=3,4$ it was shown by Edouard Collignon \cite{collignon} in 1906 that the equation does not
have a solution in positive integers. A variant of this equation for $\ell=3$ is related to cube dissections \cite{frederickson}. 
Here, we prove uniformly that there are no positive integer solutions for $\ell \ge 3$ : 

\begin{thm} \label{maintheorem}
Equation (\ref{gleichung}) has no solutions in positive integers $k,n$ for any $\ell \ge 3$.
\end{thm}

The rest of the paper contains a proof of this result. Our ideas are somewhat inspired by Runge's power series 
method \cite{runge}, for which quantitative versions are known \cite{walsh}. 
We would like to thank Harald Scheid for mentioning this diophantine problem, and Greg Frederickson 
for pointing out the history behind it.  
Our work has led to entry A234319 in the Online Encyclopedia of Integer Sequences (www.oeis.org), where the reader can find more
references.

\section{Proof of theorem \ref{maintheorem}}

To prove the theorem, we use the new variable $w=n+k$ and seek solutions in $k$ and $w$. 
Thus, equation (\ref{gleichung}) is equivalent to 

\begin{equation} \label{gleichung1}
w^\ell=\sum_{i=1}^k \left((w+i)^\ell - (w-i)^\ell \right) = \sum_{m=0}^\ell \binom{\ell}{m} w^{\ell-m} (1-(-1)^m) \sum_{i=1}^k i^m.
\end{equation}

Only summands for $m$ odd occur on the right side. Hence, we may rewrite equation \eqref{gleichung} as 

\begin{equation} \label{gleichung2}
w^\ell = 2 \sum_{m \text{ odd}} \binom{\ell}{m} w^{\ell-m} \sum_{i=1}^k i^m.
\end{equation}

If $\ell$ is even, all summands are divisible by $w$, since $\ell-m \ge 1$. 
Hence, we replace equation \eqref{gleichung} for even $\ell \ge 2$ by

\begin{equation} \label{gleichung3}
w^{\ell-1} = 2 \sum_{m \text{ odd}} \binom{\ell}{m} w^{\ell-m-1} \sum_{i=1}^k i^m.
\end{equation}

For $\ell=1$ and $\ell=2$ equations \eqref{gleichung2} resp. \eqref{gleichung3} are of the form 
$w=k(k+1)$ resp. $w=2k(k+1)$, and give exactly the solutions mentioned in the introduction.   
The Carlitz-von Staudt theorem \cite{moree} states that 

\begin{align*}
\sum_{i=1}^{k} i^m \equiv 0 \mod \frac{k(k+1)}{2}.
\end{align*}

for $m$ odd. Therefore, $k(k+1)$ is a divisor of $2\sum_{i=1}^{k} i^m$, hence of $w^\ell$.
Hence ${\rm rad}(k(k+1))$ divides $w$, and in particular $w$ is even.
(Recall that ${\rm rad}(x)$ is, by definition, the product of all positive primes dividing an integer $x$.) \\
Now we show that there is at most one solution $w=w(k)$ for any positive integer $k$ and any $\ell \ge 1$. 

\begin{prop}
For any $k,\ell \ge 1$ there is at most one positive $w=w(k)$ solving the equations \eqref{gleichung2} or \eqref{gleichung3}. 
\end{prop}

\begin{proof} By subtracting the right-hand side from the left-hand side in both equations 
\eqref{gleichung2} and \eqref{gleichung3}, one obtains a polynomial $f(k,w)$ with 

\begin{align*}
f(k,w)=& w^\ell-2 \sum_{m \text{ odd}} \binom{\ell}{m} w^{\ell-m} \sum_{i=1}^k i^m &\text{ for } \ell \text{ odd, }\\
f(k,w)=& w^{\ell-1}-2 \sum_{m \text{ odd}} \binom{\ell}{m} w^{\ell-m-1} \sum_{i=1}^k i^m  &\text{ for } \ell \text{ even. }
\end{align*}

In each case, $f$ has only one sign change. Therefore, Descartes' rule of signs 
implies that for each $k \ge 1$ there is at most one positive $w=w(k)$ solving the equations.
\end{proof}


For the remaining arguments, we define integers $e,f,g$ by the $2$-adic valuations
\[
f:=\nu_2(k(k+1)), \quad g:=\nu_2(w), \quad e:=\nu_2(\ell).
\]
In the following, we will often write $2^f \mid \mid k(k+1)$ etc. instead. 

By a theorem of MacMillan and Sondow \cite[Thm. 1]{macson}, the power sums have $2$-adic valuations 
\[
\nu_2(2 \sum_{i=1}^k i^m)=2f-1, 
\]
independent of $m$ for all odd $m \ge 3$ and all $k \ge 1$.

\begin{prop} One has 
$g \ge e+1$ for all solutions $w$ and all $\ell$. 
\end{prop}

\begin{proof} We may assume that $\ell$ is even, since otherwise $e=0$.
The integer $w$ is even, as we have seen already. Let $p$ be any odd prime dividing $w+1$. 
Equation \eqref{gleichung} is equivalent to 
\[
\sum_{i=0}^k (w-i)^\ell = \sum_{i=1}^k  (w+i)^\ell. 
\]
This implies that 
\[
\sum_{i=0}^k (w-i)^\ell - \sum_{i=1}^k  (w+i)^\ell \equiv k^\ell+(k+1)^\ell \equiv 0 \mod p.
\]
This gives that $k+1 \not \equiv 0 \mod p$, since otherwise $k+1$ and $k$ would be both divisible by $p$.
We get
\[ 
-1 \equiv \left(\frac{k}{k+1}\right)^\ell \mod p. 
\] 
Let $\zeta$ be a primitive root modulo $p$. Then, for some integer $a$, we have 
$\zeta^{a2^e} \equiv -1 \mod p$, hence $a 2^e \equiv \frac{p-1}{2} \mod p-1$.
Let $h$ be such that $2^h \mid \mid p-1$. Then $2^{h-1} \mid \mid a 2^e$. Therefore,
$h-1 \ge e$ and $2^{e+1} \mid p-1$. Thus $p \equiv 1 \mod 2^{e+1}$. Since we proved this for all 
odd primes $p$ dividing $w+1$, and $w+1$ is odd, we conclude that 
$w+1 \equiv 1 \mod 2^{e+1}$. This shows that $g \ge e+1$. 
\end{proof}

\begin{prop} \label{mainprop} 
For every solution $(k,w(k))$ in non-zero positive integers and every $\ell \ge 1$ one has 
\[
\ell k(k+1) +\frac{(\ell-1)(\ell-2)}{12 \ell} - \frac{(\ell-1)^2(\ell-2)^2}{72 \ell^3 k(k+1)} 
\]
\[
\le w(k) \le \ell k(k+1) +\frac{(\ell-1)(\ell-2)}{12 \ell}. 
\]
\end{prop}

\begin{proof} We may assume $\ell \ge 3$ and we use the abbreviations 
\[
K=k(k+1), \quad a=\frac{(\ell-1)(\ell-2)}{12 \ell}
\text{ and } b=\frac{(\ell-1)^2(\ell-2)^2}{72 \ell^3}.
\] 
Therefore, we have $b=\frac{2a^2}{\ell}$. It is easy to see that
\[
\left(1-\frac{2a}{\ell K}\right)^3 < 8 \left(1-\frac{a}{\ell K}\right),
\]
since $\frac{a}{\ell K}$ is small. By an easy computation (see appendix~\ref{app1}), this is equivalent to 

\[
\left(\ell K + a-\frac{b}{K}\right)^3 < \ell K \left(\ell K + a-\frac{b}{K}\right)^2 + \ell^2 a K^2.
\]

Let $w_0:=\ell K+a-\frac{b}{K}$. Then 
\[
w_0^3< \ell K w_0^2 +\ell^2 a K^2, 
\]
or, equivalently,
\[
w_0^\ell< \ell K w_0^{\ell-1} + \ell^2 a K^2 w_0^{\ell-3}, \, \text{ for } \ell \ge 3 \text{ odd,}
\]
\[
w_0^{\ell-1}< \ell K w_0^{\ell-2} + \ell^2 a K^2 w_0^{\ell-4}, \, \text{ for } \ell \ge 4 \text{ even.}
\]
The terms on the right-hand side are the terms for $m=1$ and $m=3$ in equations \eqref{gleichung2} and \eqref{gleichung3}. 
All other terms are also of the same sign, therefore we conclude that $f(k,w_0)<0$. This shows that $w_0$ is a lower bound. 

To obtain the upper bound, the idea is to divide equation \eqref{gleichung2} by $w^{\ell-1}$, 
and use the lower bound in all occurences of $w$ in the denominator.
The lower bound implies that
\[
w = \ell K +a - \frac{a(\ell-1)(\ell-2)}{6\ell^2K} \ge \ell K +\frac{11a}{12}  \ge \ell K
\]
for all $\ell$, since $K \ge 2$. If $\ell=3$ or $\ell=4$, then equation \eqref{gleichung2} gives
\[
w=\ell K + \frac{\ell^2 a K^2}{w^2} \le \ell K +a,
\]
and we are done. For $\ell \ge 5$ we write equation \eqref{gleichung2} as
\[
w=2 \sum_{m \text{ odd}} \binom{\ell}{m} \frac{\sum_{i=1}^k i^m}{w^{m-1}}= 
\ell K + \frac{\ell^2 a K^2}{w^2} + 2 \sum_{m \ge 5  \text{ odd}} \binom{\ell}{m} \frac{\sum_{i=1}^k i^m}{w^{m-1}}, 
\]
and estimate using $w \ge \ell K+\frac{11a}{12} \ge \ell K$
\begin{align*}
w & \le \ell K + \frac{\ell^2 a K^2}{(\ell K +\frac{11a}{12})^2} + 2 \sum_{m \ge 5 \text{ odd}} \binom{\ell}{m} \frac{\sum_{i=1}^k i^m}{(\ell K)^{m-1}}\\
& \le \ell K + \frac{a}{(1+\frac{11a}{12K\ell})^2} +  2 \sum_{m \ge 5 \text{ odd}} \binom{\ell}{m} \frac{k^m \sum_{i=1}^k (\frac{i}{k})^m}{(\ell K)^{m-1}}.
\end{align*} 

For $m \ge 5$, we use $\sum_{i=1}^k (\frac{i}{k})^m \le \sum_{i=1}^k (\frac{i}{k})^3 = \frac{K^2}{4k^3}$, and 
\[
\binom{\ell}{m} \ell^{1-m} \le \frac{(\ell-1)(\ell-2) \ell^{m-2} \ell^{1-m}}{m!}=\frac{(\ell-1)(\ell-2)}{\ell m!},
\]
so that we get

\begin{align*}
w & \le \ell K + \frac{a}{(1+\frac{11a}{12K\ell})^2}+ \frac{(\ell-1)(\ell-2)K^3}{2\ell k^3}  \sum_{m \ge 5  \text{ odd}} \frac{1}{(k+1)^m m!} \\
& \le \ell K + \frac{a}{(1+\frac{11a}{12K\ell})^2} + \frac{(\ell-1)(\ell-2)K^3}{2 \ell k^3 (k+1)^5}  \sum_{m \ge 5  \text{ odd}} \frac{1}{m!} \\
& \le \ell K + a \left( \frac{1}{(1+\frac{11a}{12K\ell})^2}+ \frac{6}{(k+1)^2} \sum_{m \ge 5} \frac{1}{m!} \right).
\end{align*}

One has
\[
\sum_{m \ge 5} \frac{1}{m!} \le  \exp(1)- 1 - 1 - \frac{1}{2} - \frac{1}{6}- \frac{1}{24} < 0.01.  
\]
For $\ell \ge 5$ one also has $\frac{a}{\ell} \ge \frac{1}{25}$. 
This implies, using Bernoulli's inequality and $K\ge 2$,
\[
 \frac{1}{(1+\frac{11a}{12 K\ell})^2} + \frac{6}{(k+1)^2} \sum_{m \ge 5  \text{ odd}} \frac{1}{m!}  
 \le \frac{1}{(1+\frac{11}{300 K})^2} + \frac{6}{K} \sum_{m \ge 5  \text{ odd}} \frac{1}{m!}
\]
\[
 \le \frac{1}{1+\frac{11}{150K}} + \frac{0.06}{K} \le \frac{K+0.06(1+\frac{11}{150K})}{K+\frac{11}{150}}
\le \frac{K+\frac{933}{15000}}{K+\frac{1100}{15000}}  < 1.
\]
Therefore, we have the upper bound $w \le \ell K +a$. 
\end{proof}

\begin{cor}
For every solution $(k,w(k))$ in positive integers and every $\ell \ge 3$ 
we have $k(k+1) < \frac{(\ell-2)^2}{12}$ and $3f+3 \le \ell$. 
\end{cor}

\begin{proof} We know that $\ell K +a-\frac{b}{K} \le w(k) \le \ell K +a$ for every positive 
integral solution $w(k)$. But one also has 
\[
\frac{\ell-3}{12} \le a < \frac{\ell-2}{12}.
\]
As $w(k)$ is an integer, we get
\[
w \le \ell K + \frac{\ell-3}{12},
\]
and thus
\[
 \frac{\ell-3}{12} \ge a-\frac{b}{K}.
\]
From this one deduces that 
\[
K \le  \frac{(\ell-1)^2(\ell-2)^2}{12\ell^2}< \frac{(\ell-2)^2}{12}.
\]
This implies $\ell-2 > \sqrt{12 K}$. As $K \ge 2^f$, and $12 \cdot 2^f \ge 9 \cdot f^2$ for all $f \ge 1$, 
we have $\ell-2>3f$, and the claim follows. 
\end{proof}

\begin{proof}[Proof of Theorem \ref{maintheorem}]
The idea of the proof is to look at $2$-adic valuations in the terms of the equation
together with the above inequalities. 

First let $\ell$ be even. We may assume that $\ell \ge 6$ since $f \ge 1$ and $\ell \ge 3f+3$. 
Look at equation \eqref{gleichung3} modulo the integer $s=2^{(2f-1)+2g+e}$. 

We claim that all summands on the right-hand side for odd $m$ with 
$3 \le m \le \ell-3$ are $\equiv 0 \mod s$. This is true for $m=\ell-3$, since then 
$2^{2g} \mid w^{2}$ and one also has 
$2^e \mid \binom{\ell}{\ell-3}=\frac{\ell(\ell-1)(\ell-2)}{1 \cdot 2 \cdot 3}$ and $2^{2f-1} \mid 2 \sum_{i=1}^k i^{\ell-3}$. 

For $m$ odd with $3 \le m \le \ell-5$, one has as well $2^{2f-1} \mid \mid 2 \sum_{i=1}^k i^m$, and in addition 
$2^{g(\ell-m-1)} \mid \mid w^{\ell-m-1}$. Using $g \ge e+1$, one has $g(\ell-m-1) \ge 4g \ge 2g+e$ and therefore the assertion follows. 
Therefore, we obtain

\begin{align*}
w^{\ell-1} \equiv \ell w^{\ell-2} k(k+1) + 2 \ell   \sum_{i=1}^k i^{\ell-1} \mod s.  
\end{align*}

We have shown that $2^{e+2f-1} \mid \mid  2 \ell   \sum_{i=1}^k i^{\ell-1}$. Since $w$ is even, we have $g \ge 1$.
Therefore, $e+2f-1<(2f-1)+2g+e$. We also have $2^{e+(\ell-2)g+f} \mid  \mid \ell w^{\ell-2} k(k+1)$. Using $\ell \ge 3f+3$
we get 
\[
e+(\ell-2)g+f \ge e+(3f+1)g +f \ge (3f-1)g+2g+e \ge (2f-1)+(2g+e).
\]
Hence the first term on the right is also $\equiv 0$ modulo $s$. 
This implies $(\ell-1)g = e+2f-1$. This is a contradiction, using $\ell \ge 3f+3$ and $g \ge e+1$. \\
Assume now that $\ell \ge 5$ is odd (hence $e=0)$, and look at equation \eqref{gleichung2} modulo $s=2^{(2f-1)+2g}$.
Again, all summands on the right-hand side for odd $m$ with $3 \le m \le \ell-2$ are $\equiv 0 \mod s$, since
$2^{2f-1} \mid \mid 2 \sum_{i=1}^k i^m$, and $g(\ell-m) \ge 2g$ for $m \le \ell-2$.
Therefore, we obtain

\begin{align*}
w^{\ell} \equiv  \ell w^{\ell-1} k(k+1) + 2 \sum_{i=1}^k i^{\ell} \mod s.  
\end{align*}

We know that $2^{2f-1} \mid \mid  2 \sum_{i=1}^k i^{\ell}$. Since $w$ is even, we have $g \ge 1$.
Therefore, $2f-1<(2f-1)+2g$. We also have $2^{(\ell-1)g+f} \mid  \mid \ell w^{\ell-1} k(k+1)$. 
Using $\ell \ge 3f+3$ we get 
\[
(\ell-1)g+f \ge (3f+2)g+f \ge (2f-1)+2g
\]
and hence the first term on the right is also $\equiv 0$ modulo $s$. 
This implies $\ell g=2f-1$. This is a contradiction, using $\ell \ge 3f+3$. 
\end{proof}

\appendix 

\section{~} \label{app1} 

Here we supply the missing computation from the proof of Prop.~\ref{mainprop}:

\begin{align*}
& \left(1-\frac{2a}{\ell K}\right)^3 < 8 \left(1-\frac{a}{\ell K}\right) \\
\Leftrightarrow & \; a \left(1-\frac{2a}{\ell K}\right)^3 - 8a \left(1-\frac{a}{\ell K}\right) <0 \\
\Leftrightarrow & \left(1-\frac{2a}{\ell K}\right)^2 \left(2 \ell K +a\left(1 - \frac{2a}{\ell K}\right) \right) <2 \ell K \\
\Leftrightarrow & \left(a-\frac{b}{K}\right)^2 \left(2 \ell K +a - \frac{b}{K} \right)  <  b\ell^2 K \\
\Leftrightarrow & \left(a-\frac{b}{K}\right) \left(2 \ell K\left(a-\frac{b}{K}\right)  +\left(a-\frac{b}{K}\right)^2 \right)  <  b \ell^2 K  \\
\Leftrightarrow & \left(a-\frac{b}{K}\right)  \left(\ell^2 K^2 + 2 \ell K \left(a-\frac{b}{K}\right) + \left(a-\frac{b}{K}\right)^2 \right)  <  \ell^2 a K^2\\
\Leftrightarrow & \left(a-\frac{b}{K}\right) \left(\ell K + a-\frac{b}{K}\right)^2  <  \ell^2 a K^2 \\
\Leftrightarrow & \left(\ell K + a-\frac{b}{K}\right)^3 < \ell K \left(\ell K + a-\frac{b}{K}\right)^2 + \ell^2 a K^2.
\end{align*}

\end{document}